\documentclass[12pt]{amsart}

\usepackage{amsmath, amsfonts, amssymb}
\usepackage{xcolor}

\newcommand{\R}{\mathbb{R}}
\newcommand{\C}{\mathbb{C}}
\newcommand{\T}{\mathbb{T}}
\newcommand{\D}{\mathbb{D}}
\newcommand{\B}{\mathbb{B}}

\newtheorem{theorem}{Theorem}
\newtheorem{proposition}[theorem]{Proposition}

\theoremstyle{definition}

\newtheorem{definition}[theorem]{Definition}
\newtheorem{example}[theorem]{Example}
\newtheorem{question}{Open Problem}

\DeclareMathOperator{\BMO}{BMO}

\DeclareMathOperator{\IMO}{IMO}
\DeclareMathOperator{\IDA}{IDA}

\DeclareMathOperator{\rank}{rank}
\DeclareMathOperator{\linspan}{span}

\begin{document}

\author[Z. Hu]{Zhangjian Hu}
\address{Department of Mathematics, Huzhou University, Huzhou, Zhejiang, China}
\email{huzj@zjhu.edu.cn}

\author[J. A. Virtanen]{Jani A. Virtanen}
\address{Department of Mathematics and Statistics, University of Reading, Reading, England}
\email{j.a.virtanen@reading.ac.uk}

\title{On the Berger-Coburn phenomenon}



\thanks{Z.~Hu was supported in part by the National Natural Science Foundation of China (12071130, 12171150) and J.~Virtanen was supported in part by Engineering and Physical Sciences Research Council grant EP/T008636/1.}

\begin{abstract}
In their previous work, the authors proved the Berger-Coburn phenomenon for compact and Schatten $S_p$ class Hankel operators $H_f$ on generalized Fock spaces when $1<p<\infty$, that is, for a bounded symbol $f$, if $H_f$ is a compact or Schatten class operator, then so is $H_{\bar f}$. More recently J.~Xia has provided a simple example that shows that there is no Berger-Coburn phenomenon for trace class Hankel operators on the classical Fock space $F^2$. Using Xia's example, we show that there is no Berger-Coburn phenomena for Schatten $S_p$ class Hankel operators on generalized Fock spaces $F^2_\varphi$ for any $0<p\le 1$. Our approach is based on the characterization of Schatten class Hankel operators while Xia's approach is elementary and heavily uses the explicit basis vectors of $F^2$, which cannot be found for the weighted Fock spaces that we consider. We also formulate four open problems.
\end{abstract}

\maketitle

\section{Introduction}
Given a weight $\varphi : \C^n\to \R_+$, we say that a measurable function $f: \C^n\to \C$ is in $L^2_\varphi$ if
$$
	\|f\|_{\varphi} = \left(\int_{\C^n} |f(z)|^2 e^{-2\varphi(z)} dv(z)\right)^{1/2} < \infty,
$$
where $v$ is the standard Lebesgue measure on $\C^n$. For an open set $\Omega\subset\C^n$, we denote by $H(\Omega)$ the space of all holomorphic functions on $\C^n$. The \emph{weighted Fock space} $F^2_\varphi$ is defined by
$$
	F^2_\varphi = H(\C^n) \cap L^2_\varphi.
$$
In this work, we consider the weights $\varphi\in C^2(\C^n)$ that satisfy the condition that there are two positive constants $m$ and $M$ such that
\begin{equation}\label{weights-a}
	m\omega_0 \le \mathrm{i} \partial \overline{\partial} \varphi \le M\omega_0
\end{equation}
in the sense of currents,  where $\omega_0 = \mathrm{i} \partial \overline{\partial}  |z|^2 $ is the Euclidean-K\"{a}hler form on $\C^n$. The expression (\ref{weights-a}) is also denoted by $\mathrm{i} \partial \overline{\partial} \varphi \simeq \omega_0$, and when $n=1$, it simplifies to the form
$$m
	 \le \Delta \varphi \le M,
$$
where $\Delta\varphi$ is the Laplacian of $\varphi$.

Notice that the standard radial weights $\varphi(z) = \alpha|z|^2$ with $\alpha>0$ satisfy~\eqref{weights-a}. For its role in a number of important developments in the theory of Fock spaces, we single out the classical Fock space $F^2$ of standard Gaussian square-integrable entire functions, which corresponds to $\alpha=1$. Also, each Fock-Sobolev space $F^{2,m}$  consisting of entire functions $f$ for which $\partial^{\alpha} f \in F^2 $ for all multi-indices $|\alpha|\le m$ can be viewed as a $F^2_\varphi$ with some $\varphi$ satisfying \eqref{weights-a}.

In this note we are concerned with \emph{Hankel operators} $H_f$ defined for symbols $f\in L^\infty(\C^n)$ by
\begin{equation}\label{e:T and H}
	H_fg = (I-P)(fg)
\end{equation}
for $g\in F^2_\varphi$. Clearly $H_f : F^2_\varphi \to L^2_\varphi$ is a bounded operator if $f\in L^\infty$. Denote by $H^2$ the Hardy space on the unit circle $\T$ and by $A^2(\B_n)$ the Bergman space on the unit ball
$$
	\B_n = \left\{ z\in \C^n : \|z\|^2 = \langle z,z \rangle = z_1\bar z_1 + \ldots + z_n\bar z_n < 1\right\},
$$
that is,
$$
	H^2 = \{f\in L^2(\T) : f_k = 0 \ {\rm for}\ k<0\},
$$
where $f_k$ is the $k$th Fourier coefficient of $f$, and
$$
	A^2 = \left\{f\in H(\B_n) : \int_{\B_n} |f(z)|^2\, dv(z)<\infty \right\}.
$$
We denote by $P$ both the orthogonal projection (aka the Riesz projection) of $L^2(\T)$ onto $H^2$ and the orthogonal projection (aka the Bergman projection) of $L^2(\B_n)$ onto $A^2$. Given a bounded symbol $f$, we define the Hankel operator $H_f$ as in \eqref{e:T and H}.

Of particular interest is the following property, which we refer to as the Berger-Coburn phenomenon for an operator ideal $\mathcal{M}$.

\begin{definition}
Let $P$ an orthogonal projection on a Hilbert space $H$ and $A$ be a linear operator on $H$. Define the (abstract) Hankel operator $H_A : P(H) \to H$ by
$$
	H_A = (I-P)A.
$$
Given an operator ideal $\mathcal{M}$ in $B(P(H), H)$ and a set $\mathcal{A}$ of linear operators $A: H\to H$, we say that the \emph{Berger-Coburn phenomenon} (BCP) holds for $(\mathcal{M}, \mathcal{A})$ if
\begin{equation}\label{e:BCP}
	H_A \in \mathcal{M} \iff H_{A^*} \in \mathcal{M}
\end{equation}
for every $A\in \mathcal{A}$, where $A^*$ is the adjoint of $A$.
\end{definition}

In what follows, we only consider the Hankel operators $H_f = H_{M_f}$ defined above acting from the Fock space, the Bergman space, or the Hardy space to the corresponding $L^2$ space when $\mathcal{A} = \{M_f : f\in L^\infty\}$. However, we have given the preceding definition in a quite general context as it may be of interest in more abstract settings. Observe that in these three cases $M^*_f=M_{\bar f}$ so that \eqref{e:BCP} is equivalent to
$$
	H_f\in \mathcal{M} \iff H_{\bar f}\in \mathcal{M}
$$
for every $f\in L^\infty$.

Among the most important operator ideals are those of compact operators and of Schatten class operators. To define these for bounded linear operators $T : H_1 \to H_2$ between two Hilbert spaces, we denote by $s_j(T)$ the singular values of $T$, that is,
\begin{equation}\label{e:singular values}
	s_j(T) = \inf \{ \|T-K\| : K : H_1\to H_2,\ \rank K\le j\},
\end{equation}
where $\rank K$ denotes the rank of $K$. Recall that the operator $T$ is said to be compact if and only if $s_j(T) \to 0$, and we write $K(H_1, H_2)$ for the operator ideal of all compact operators acting from $H_1$ to $H_2$. For $0<p<\infty$, we say that $T$ is in the Schatten class $S_p$ and write $T\in S_p(H_1, H_2)$ if $\|T\|_{S_p}^p = \sum_{j=0}^\infty \left(s_j(T)\right)^p < \infty$, which defines a norm when $1\le p<\infty$ and a quasinorm otherwise. Note that $S_p$ are also called the Schatten-von Neumann classes or trace ideals; for further details, see, e.g., \cite{Simon}.

\section{Berger-Coburn phenomenon for compact operators}

Let $f\in L^\infty$. In 1987, Berger and Coburn~\cite{BC1987} proved that the BCP holds for $K(F^2, L^2(\C^n, e^{-|z|^2}dv))$ using heavy machinery of $C^*$-algebras and Hilbert space theory. A few years later, a more elementary proof of this result was given by Stroethoff~\cite{St92}, who derived it as a consequence of his characterization of compactness, that is, $H_f$ is compact if and only if
\begin{equation}
	\|(I-P)(f\circ \tau_\lambda)\|_{L^2(\C^n)} \to 0\ {\rm as}\ |\lambda|\to \infty,
\end{equation}
where $\tau_\lambda(z) = z+\lambda$ is the translation.

One reason which makes Berger and Coburn's compactness result so striking is that analogous statements in the settings of the Hardy space and the Bergman space fail as seen in the following examples.

\begin{example}\label{examples for Hardy and Bergman}
Let $b$ be the Blaschke product
$$
	b(z) = \prod_{k=1}^\infty \frac{\alpha_k-z}{1-\bar \alpha_k z}\frac{|\alpha_k|}{\alpha_k}
	\quad (z\in\D),
$$
where $\alpha_k = 1-\frac1{2^k}$. Because $b$ is a bounded analytic function, $H_b = 0$ is trivially compact, but it can be shown that $\bar b$ is not in the little Bloch space and hence $H_{\bar b}$ is not compact according to Axler's characterization of compact Hankel operators with conjugate analytic symbols (see pp.~318--319 of~\cite{HaggerV} and the related references therein). Thus, there is no BCP for $K(A^2, L^2(\D))$. Another example in $A^2$ to the same effect can be obtained from the bounded analytic function $f(z)=\exp\left(\frac{z+1}{z-1}\right)$ defined for $z\in\D$.

 The counterexample $b$ above can be modified to show that there is no BCP for $K(A^2, L^2(\B_n))$. In fact, we can simply extend $b$ by setting $b(z) = b(z_1)$ for $z=(z_1, z_2, \cdots, z_n)\in \B_n$. Then $b$ is a bounded holomorphic function on $\B_n$ and $H_b=0$. On the other hand, 
$$
	\lim_{|z|\to 1}(1-|z|)\frac{\partial b(z)}{\partial z_1} \neq 0
$$ 
so that  $b$ is not in the little Bloch space as Theorem 16 in \cite{Zh98} shows. Now Theorem~E of \cite{Li92} implies that $H_{\bar b}\notin K(A^2, L^2(\B_n))$.

To see that there is no BCP for $K(H^2, L^2(\T))$, we can choose a function $f$ in $H^\infty\setminus VMO$ (using a result of Wolff~\cite{Wolff} on the class $QA = VMO\cap H^\infty$), so that $H_f=0$ is compact while $H_{\bar f}$ cannot be compact because of the characterization that states $H_f$ and $H_{\bar f}$ are both compact if and only if $f\in VMO$.
\end{example}

It may be reasonable to ask the following question.

\begin{question}\label{q1}
Are there nontrivial operator ideals in 
$$
	B(H^2(\T), L^2(\T))\quad\text{\rm or in}\quad B(A^2(\B_n), L^2(\B_n))
$$
for which the Berger-Coburn phenomenon holds?
\end{question}

We now turn our attention back to Fock spaces. While we already know this from the result of Berger and Coburn, it is still worth pointing out that examples in the spirit of Example~\ref{examples for Hardy and Bergman} are doomed to fail because if $f$ is a bounded entire function, then it is constant and hence $H_{\bar f}$ is also compact. In fact, this is an indication of why there is BCP for $K(F^2, L^2(\C^n, e^{-|z|^2}dv)$ but not for $K(A^2, L^2(\B_n))$. A complete explanation is given by the following characterization of Hagger and Virtanen~\cite{HaggerV} (obtained using limit operator techniques): Let $\Omega\in \{\B_n, \C^n\}$, $X^2=A^2_\alpha$ if $\Omega=\B_n$, $X^2=F^2_\alpha$ if $\Omega=\C^n$, and denote by $\beta \Omega$ the Stone-\v{C}ech compactification of $\Omega$. Then $H_f : X^2_\alpha \to L^2_\alpha(\Omega)$ is compact if and only if for every $x\in \beta\Omega\setminus\Omega$ there is a $h_x\in H^\infty(\Omega)$ such that for all nets $(z_\gamma)$ in $\Omega$ converging to x, we have
\begin{equation}\label{e:HaggerV}
	\|f\circ \phi_{z_\gamma} - h_x\|_{L^p_\alpha} \to 0 \ {\rm as}\ z_\gamma\to x,
\end{equation}
where $\phi_z(w) = w-z$ if $\Omega=\C^n$ and $\phi_z$ is the (unique) geodesic symmetry interchanging $0$ and $z$ if $\Omega = \B_n$. Now, by Liouville's theorem, the condition in~\eqref{e:HaggerV} is symmetric in $f$ and $\bar f$ if $\Omega=\C^n$, and hence $H_f : F^2_\alpha \to L^2_\alpha$ is compact if and only if $H_{\bar f}$ is compact.

Finally we mention that a fourth proof of Berger and Coburn's result can be found in \cite{HuVi3} (see Theorem~1.2), which treats the most general setting currently known in terms of weights $\phi$ (and even proves that $H_f : F^p_\phi \to L^q_\phi$ is compact if and only if $H_{\bar f}$ is compact when $0<p\le q<\infty$ or $1\le q<p<\infty$). While the weights considered in \cite{HuVi3} satisfy a slightly stronger condition than \eqref{weights-a}, we note that we can prove the case $p=q=2$ for the weights satisfying \eqref{weights-a}. 

\section{Berger-Coburn phenomenon for Schatten class operators}

A question of whether the Berger-Coburn phenomenon holds for the Schatten classes $S_p$ arises as a natural next step of the work in the previous section. In the case of $H^2$ or $A^2$, we get a negative answer immediately from Example~\ref{examples for Hardy and Bergman}.

Next we consider Hankel operators on Fock spaces. In 2004, for Hilbert-Schmidt Hankel operators on $F^2$, the question was answered in the affirmative by Bauer~\cite{Bauer} using an approach similar to \cite{St92}. Further, on page 1384 of \cite{XZ2004} (see Note added January 31, 2003), Xia and Zheng state that their preprint in September 2000 contained Bauer's result in the one-dimensional case. They further comment that because their methods were ``rather ad hoc'' and the preprints of Bauer \cite{Bauer} (April 2002) and Stroethoff (September 2002, unpublished) proved the result in $\C^n$, they decided not to include their result in the final version of \cite{XZ2004}. Their note described the remaining cases as ``rather challenging.'' Indeed, no further progress was made until the following result of Hu and Virtanen~\cite{HuVi4, HuVi4c}.

\begin{theorem}\label{p>1}
Let $\varphi\in C^2(\C^n)$ satisfy the condition in~\eqref{weights-a}. If $1<p<\infty$, then the {\rm BCP} holds for $S_p(F^2_\varphi, L^2_\varphi)$, that is, for $f\in L^\infty$, $H_f\in S^p$ if and only $H_{\bar f}\in S_p$ with the $S_p$-norm estimate
$$
	\|H_{\bar f}\|_{S_p} \le C \|H_f\|_{S_p},
$$
where the constant $C$ is independent of $f$.
\end{theorem}

The case $p=1$ for the classical Fock space was recently solved by Xia~\cite{Xia2023}, who showed that if
\begin{equation}\label{e:ex}
	f(z) = \begin{cases}
	\frac1z & {\rm if}\ |z|\ge 1\\
	0 & {\rm if}\ |z|<1
	\end{cases}
\end{equation}
then $H_f : F^2(\C) \to L^2(\C, e^{-|z|^2}\, dv)$ is in the trace class but $H_{\bar f}$ is not, so there is no BCP for $S_1(F^2(\C), L^2(\C, e^{-|z|^2}\, dv))$. We note that Xia's work cannot be extended to weighted Fock spaces $F^2_\varphi$.

We present an alternate approach that works for weights satisfying~\eqref{weights-a}. Indeed, applying the characterizations of Schatten class Hankel operators of Hu and Virtanen \cite{HuVi4}, we show that Xia's example can be used to prove that there is no BCP for any $S_p(F^2_\varphi, L^2_\varphi)$ with $0<p\le 1$. To accomplish this, we need some preliminaries.

For $f\in L^2_{\mathrm{loc}}$ (the set of all locally square integrable functions on $\C^n$) and $0<q<\infty$, define
\begin{equation}\label{G-q-r}
	G_{q,r}(f)(z)=\inf_{h\in H(B(z,r))} \left(\frac{1}{|B(z,r)|}\int_{B(z,r)} |f-h|^q dv \right)^{\frac{1}{q}}\quad (z\in \C^n).
\end{equation}
For $0< s\le \infty$, the space $\mathrm{IDA}^{s,q}$ (Integral Distance to Analytic Functions) consists of all $f\in L^2_{\mathrm {loc}}$ such that
\begin{equation}\label{IDA definition}
	\|f\|_{\mathrm{IDA}^{s,q}} =  \|  G_{q,r}(f)\|_{L^s}<\infty
\end{equation}
for some $r>0$. When $1\le q<\infty$ and $0<s<\infty$, notice that in $\IDA^{s,q}$ the parameter $r$ is suppressed because the space is independent of $r$ and different values of $r$ give equivalent norms on each space (see Remark~3.8 of~\cite{HuVi3}).

We remark that the $\IDA^{p,2}$ type conditions were first introduced by Luecking~\cite{Luecking1992} in his study of Schatten class Hankel operators on the Bergman space $A^2$ on the unit disk.

It may be helpful to compare the spaces $\IDA^{s,q}$ with the spaces $\BMO^p$ of functions of bounded mean oscillation. Recall the space $\BMO^p$ is defined by
$$
	\BMO^p = \left\{ f\in L^p_{\rm loc} : \sup_{z\in\C^n} MO_{p,r} (f)(z) <\infty\right\},
$$
where
$$
	MO_{p,r}(f)(z) = \left( \frac1{|B(z,r)|} \int_{B(z,r)} |f - \widehat f_r(z)|^p\, dv\right)^{1/p}
$$
and $\widehat f_r$ is the average function defined on $\C^n$ by
$$
	\widehat f_r(z) = \frac1{|B(z,r)|} \int_{B(z,r)} f\, dv.
$$
It is well known that the space $\BMO^p$ is independent of $r$.

\medskip

We can now single out a large class of entire functions in $\IDA^{s,q}$ that are not in $\BMO^p$.

\begin{proposition}
If $f\in H(\C)$ and $f'$ is not constant, then
$$
	f\in \IDA^{s,q}\setminus \BMO^p
$$
whenever $1\le p,q<\infty$ and $0<s\le \infty$.
\end{proposition}
\begin{proof}
Since $f\in H(\C)$, that $f\in \IDA^{s,q}$ for all $0<q<\infty$ and $0<s<\infty$ is an trivial consequence of that $G_{q,r}(f)\equiv 0$. On the other hand, it follows from Lemma~2.1 of~\cite{Luecking} that there is a positive constant $C$ depending only on $p$ and $r$ such that
$$
	|f'(z)| \le C\left( \frac1{|B(z,r)|} \int_{B(z,r)} |f-f(z)|^p dv\right)^\frac1p.
$$
Further, since $f'' \ne 0$,
$$
	\limsup_{z\to\infty} |f'(z)|^s = \infty,
$$
and so
$$
	\limsup_{z\to \infty} MO_{p,r}(f)(z)^s = \infty,
$$
which implies that $f\notin \BMO^p$.
\end{proof}

The following result characterizes Schatten class Hankel operators in terms of $\IDA^p$ and it is contained in Theorem~1.1 of \cite{HuVi4}, where $\IDA^p=\IDA^{p,2}$. A similar characterization for Hankel operators on the (unweighted) Bergman space $A^2(\D)$ to be in Schatten classes $S_p$ when $1\le p<\infty$ was obtained by Luecking~\cite{Luecking1992}.

\begin{theorem}\label{char1}
Let $f\in L^\infty$, $\varphi\in C^2(\C^n)$ be real valued with $\mathrm{i} \partial \overline{\partial} \varphi \simeq \omega_0$, and $0<p<\infty$. Then $H_f \in S_p(F^2_\varphi, L^2_\varphi)$ if and only if $f\in \IDA^p$.
\end{theorem}

We also need a result on simultaneous Schatten class membership of Hankel operators. For this purpose, for $0<s\le\infty$, we define the space $\IMO^s$ (of functions of integral mean oscillation) by
\begin{equation}
	\IMO^s = \left\{ f\in L^2_{\rm loc} : MO_{2,r}(f) \in L^s\right\}.
\end{equation}
Observe that the space $\IMO^s$ is independent of $r$. The following result is contained in Theorem~6.2 of \cite{HuVi4}. 

\begin{theorem}\label{char2}
Let $f\in L^\infty$, $\varphi\in C^2(\C^n)$ be real valued with $\mathrm{i} \partial \overline{\partial} \varphi \simeq \omega_0$, and $0<p<\infty$. Then both $H_f$ and $H_{\bar f}$ are in ${S_p}(F^2_\varphi, L^2_\varphi)$ if and only if $f\in \IMO^p$.
\end{theorem}

Related to the preceding theorem, we note that previously Xia and Zheng~\cite{XZ2004} found a different characterization (in terms of ``standard deviation'') for simultaneous membership for the classical Fock space $F^2$ and the Schatten classes $S_p$ when $1\le p<\infty$.

It is also worth noting that the simultaneous Schatten class membership of $H_f$ and $H_{\bar f}$ on weighted Bergman spaces can be formulated in terms of $\IMO^p$ on the unit disk; see, e.g., Section~4 of~\cite{MR4514504} for an approach inspired by Luecking's work~\cite{Luecking1992} and the references therein for other approaches and the case of Bergman spaces on the unit ball. However, related to Open Problem~\ref{q1}, this does not seem to provide any further insight.

\medskip

We are now ready to prove our main result.

\begin{theorem}\label{p<=1}
Suppose that $\varphi\in C^2(\C)$ satisfies the condition in~\eqref{weights-a}. If $0<p\le 1$, then there is no Berger-Coburn phenomenon for the Schatten class $S_p(F^2_\varphi(\C), L^2_\varphi(\C))$, that is, there is a symbol $f\in L^\infty(\C)$ such that $H_f\in S_p$ but $H_{\bar f}\not\in S_p$.
\end{theorem}

\begin{proof}
Let $f$ be the function given in \eqref{e:ex}, that is, for $z\in \C$, $f(z) = \frac1z$ when $|z| \ge 1$ and $f(z)=0$ otherwise. As mentioned above, the definition of $\IDA^p$ is independent of $r$, and we set $r=1$ for simplicity. Also we write $C$ for positive constants which may change from line to line but do not depend on functions being considered.

Now if $|z|\ge 2$, then $f$ is holomorphic in $B(z,1)$, and hence trivially
$$
	\left(G_1(f)(z)\right)^2 = \inf_{h\in H\left(B(z,1)\right)} \frac1{|B(z,1)|} \int_{B(z,1)} |f-h|^2 dv= 0.
$$
Also it is easy to see that there is a constant $C$ such that
$$
	G_1(f)(z) < C
$$
for all $|z|<2$. Thus, $\|G_1(f)\|_{L^p} < C$, so $f\in \IDA^p$ and Theorem~\ref{char1} implies that $H_f \in S_p$.

To show that $H_{\bar f}$ is not in the Schatten class $S_p$, we apply Theorem~\ref{char2}. If $|z|\ge 2$, then $\bar f$ is harmonic in $B(z,1)$ and hence by the mean value property, we have
$$
	\widehat{\bar f_1}(z) = \frac1{|B(z,1)|} \int_{B(z,1)} \bar f dv = \overline{f(z)}.
$$
Therefore, for $|z|\ge 2$,
\begin{align*}
	\left(MO_{2,1}(f)(z)\right)^2 &= \frac1{|B(z,1)|} \int_{B(z,1)} \left| \overline{f(w)}-\overline{f(z)}\right|^2 dv(w)\\
	&=\frac1{|B(z,1)|} \int_{B(z,1)} \frac{|w-z|^2}{|wz|^2} dv(w)\\
	&\simeq \frac1{|z|^4},
\end{align*}
and so
\begin{align*}
	\int_\C (MO_{2,1}(\bar f)(z))^p\, dv(z) &\simeq \int_\C \min\left\{1, \frac1{|z|^{2p}}\right\} dv(z)\\
	&= \begin{cases}
	C<+\infty & {\rm if}\ p>1\\
	+\infty & {\rm if}\ 0< p\le 1.
	\end{cases}
\end{align*}
Thus, by Theorem \ref{char2}, $H_{\bar f}\notin S_p$ if $0<p\le 1$.
\end{proof}

We finish this section with open problems of whether the statements of Theorems~\ref{p>1} and~\ref{p<=1} and the theorem of Berger and Coburn remain true for doubling Fock spaces. Recall that a positive Borel measure $\nu$ on $\C$ is called a doubling measure if there exists some constant $M>0$ such that
$$
	\nu(D(z, 2r))\leq M\nu(D(z, r))
$$
for $z \in \C$ and $r>0$, where $D(z,r)=\left\{w\in\C : |w-z|<r \right\}$. Let $\phi$ be a subharmonic function that is not identically zero on $\C$ and suppose that $\nu= \Delta\phi\, dv$ is doubling. We call $F^2_\phi$ doubling Fock spaces. Observe that if $\varphi\in C^2(\C)$ satisfies the condition in~\eqref{weights-a}, then $F^2_\varphi$ is a doubling Fock space.

\begin{question}
Let $F^2_\phi$ be a doubling Fock space. For which values of $p$ does the Berger-Coburn phenomenon hold for $S_p(F^2_\phi, L^2_\phi)$?
\end{question}

In fact, a similar question seems open also for compactness:

\begin{question}
Let $F^2_\phi$ be a doubling Fock space and $f\in L^\infty$. Is it true that if $H_f : F^2_\phi \to L^2_\phi$ compact, then $H_{\bar f}$ is also compact?
\end{question}

\section{Unbounded symbols}

In addition to bounded symbols, Hankel operators have been widely studied for unbounded symbols. The usual way of defining them is as follows.  Suppose that $\varphi\in C^2(\C^n)$ satisfies \eqref{weights-a} and let
$$
	\Gamma = \linspan \{ K_z : z\in \C^n\},
$$
where $K_z=K(\cdot, z)$ is the reproducing kernel of $F^2_\varphi$ (see Section~2.1 of~\cite{HuVi4}), and consider the class of symbols
$$
 	\mathcal S=\left \{f \textrm{ measurable on }\C^n: f g\in L^2_\varphi\ \textrm{for}\ g \in \Gamma\right \}.
$$
Given $f \in \mathcal S$ and $g\in \Gamma$, the Hankel operator $H_f$ is well defined, and since $\Gamma$ is dense in $F^2_\varphi$, it follows that $H_f = (I-P)M_f$ is densely defined on $F^2_\varphi$. Notice that clearly $L^\infty\subset\mathcal{S}$.

The following example shows that there is no Berger-Coburn phenomenon for compact or Schatten class Hankel operators on $F^2_\varphi$ for unbounded symbols in general.

\begin{example}
Let $f(z) = z$ for $z\in\C^n$. Then $H_f=0$ is in $S_p$ for all $0<p<\infty$, but it is easy to see that $\bar f\notin \IMO^p$ for any $0<p<\infty$. In fact, we can show that $\bar f\notin \IMO^\infty$ and hence $H_{\bar f}$ cannot be compact because $H_f$ and $H_{\bar f}$ are simultaneously compact if and only if $f\in \IMO^\infty$ (see~\cite{HuVi3}).

We note that the function $f(z)=z$ defined for $z\in\C$ appeared in the example on page 2995 of~\cite{Bauer}, where it was used to show that there are unbounded functions for which $H_f\in S_2$ but $H_{\bar f}\notin S_2$. The preceding paragraph shows how this and more can be seen very easily using our characterizations, which avoids elementary calculations of~\cite{Bauer}.
\end{example}

In fact, the preceding example can be extended to all polynomials of degree at least one. We omit the details and instead show that if $f: \C\to \C$ is entire with $f'\neq 0$, then $H_f = 0\in S_p$ for all $0<p<\infty$ but $H_{\bar f}$ is not in $S_p$. Indeed,
$$
	MO_2 f(z) = \frac1{|B(z,1)|} \int_{B(z,1)} |f(w)-f(z)|^2\, dv(w) = \widehat{|f|^2}(z) - |f(z)|^2,
$$
and so $MO_2(f)\notin L^p$ and it follows that $H_{\bar f}\notin S_p$ as before.

On the other hand, we can give examples of unbounded functions for which BCP holds. Obviously, we need to consider functions that are not entire. Examples include functions $f$ on $\C^n$ so that $f\in L^2(B(0, 1))$ and $f=0$ outside $B(0, 1)$, which can be further modified so that $f$ is not bounded outside any compact set. We finish with the following general question.

\begin{question}
Characterize unbounded symbols $f$ for which BCP holds for compact operators or for Schatten class operators $H_f$ on Fock spaces.
\end{question}

\section*{Acknowledgments}

Most of this work was presented at the 33rd International Workshop on Operator Theory and its Applications (IWOTA) in Kraków by the second author, who would like to thank the main organizers, M. Ptak and M. Wojtylak, for the invitation.


\begin{thebibliography}{99}
\bibitem{Bauer} W. Bauer: Hilbert-Schmidt Hankel operators on the Segal-Bargmann space. Proc. Amer. Math. Soc. 132 (2004), no. 10, 2989--2996.

\bibitem{BC1987} C. A. Berger, L. A. Coburn: Toeplitz operators on the Segal-Bargmann space. Trans. Amer. Math. Soc. 301 (1987), no. 2, 813--829.

\bibitem{HaggerV} R. Hagger, J. Virtanen: Compact Hankel operators with bounded symbols. J. Operator Theory 86 (2021), no. 2, 317--329.

\bibitem{Li92} H. Li: BMO, VMO and Hankel operators on the Bergman space of strongly pseudoconvex domains. J. Funct. Anal. 106 (1992), no. 2, 375--408.

\bibitem{HuVi3} Z. Hu, J. Virtanen: IDA and Hankel operators on Fock spaces. Anal. PDE (in press) arXiv:2111.04821.

\bibitem{HuVi4} Z. Hu, J. Virtanen: Schatten class Hankel operators on the Segal-Bargmann space and the Berger-Coburn phenomenon. Trans. Amer. Math. Soc. 375 (2022), no. 5, 3733--3753

\bibitem{HuVi4c} Z. Hu, J. Virtanen: Corrigendum to ``Schatten class Hankel operators on the Segal-Bargmann space and the Berger-Coburn phenomenon.'' Trans. Amer. Math. Soc. (in press).

\bibitem{Luecking} D. H. Luecking: Forward and reverse Carleson inequalities for functions in Bergman spaces and their derivatives. Amer. J. Math. 107 (1985), no. 1, 85--111.

\bibitem{Luecking1992} D. H. Luecking: Characterizations of certain classes of Hankel operators on the Bergman spaces of the unit disk. J. Funct. Anal. 110 (1992), no. 2, 247--271.

\bibitem{Simon} B. Simon: Operator theory. A Comprehensive Course in Analysis, Part 4. American Mathematical Society, Providence, RI, 2015.

\bibitem{St92} K. Stroethoff: Hankel and Toeplitz operators on the Fock space. Michigan Math. J. 39 (1992), no. 1, 3--16.

\bibitem{Wolff} T. H. Wolff: Two algebras of bounded functions. Duke Math. J. 49 (1982), no. 2, 321--328.

\bibitem{Xia2023} J. Xia: The Berger-Coburn phenomenon for Hankel operators on the Fock space, Complex Anal. Oper. Theory, 17 (2023), Paper No. 35.

\bibitem{XZ2004} J. Xia, D. Zheng: Standard deviation and Schatten class Hankel operators on the Segal-Bargmann space. Indiana Univ. Math. J. 53 (2004), no. 5, 1381--1399.

\bibitem{MR4514504} Z. Zeng, X. Wang, Z. Hu: Schatten class Hankel operators on exponential Bergman spaces. Rev. R. Acad. Cienc. Exactas Fís. Nat. Ser. A Mat. RACSAM 117 (2023), no. 1, Paper No. 23, 19 pp.

\bibitem{Zh98} K. Zhu: The Bergman  spaces, the Bloch space , and
Gleason's problem, Trans. Amer. Math. Soc. 309 (1998), no. 1, 253--268.
\end{thebibliography}
\end{document}